\documentclass[preprint,12pt]{elsarticle}
\usepackage{lipsum}
\makeatletter
\def\ps@pprintTitle{%
 \let\@oddhead\@empty
 \let\@evenhead\@empty
 \def\@oddfoot{}%
 \let\@evenfoot\@oddfoot}
\makeatother

\usepackage{amssymb}
\usepackage{hyperref}
\usepackage{amsfonts}
\usepackage{latexsym}
\usepackage{amsmath}
\usepackage{tikz}
\usetikzlibrary{decorations.markings,arrows,automata}

\input xy
\xyoption{all}

\newtheorem{thm}{Theorem}

\newtheorem{lem}[thm]   {Lemma}
\newtheorem{cor}[thm]   {Corollary}
\newtheorem{rem}[thm]   {Remark}
\newtheorem{defn}[thm]  {Definition}

\newtheorem{prop}[thm]  {Proposition}

\newcounter{foo}  \Alph{foo}

\newenvironment{proof}  {\par\noindent{\bf Proof}\ }
                             {\hfill$\Box$\par\medskip}

\newcommand{\id}    {\mathrm{id}}

\newcommand{\Es}	{\mathrm{Es}}

\newcommand{\cat}   {\mathrm{cat}}

\newcommand{\gscat}    {\mathrm{gscat}}
\newcommand{\scat}    {\mathrm{scat}}

\begin{document}

\begin{frontmatter}

\title{On the Lusternik-Schnirelmann category of a simplicial map}

\author{Nicholas A. Scoville\footnote{Department of Mathematics and Computer Science, Ursinus College, Collegeville, PA  19426, {\tt nscoville@ursinus.edu}}, Willie Swei \footnote{Department of Mathematics and Computer Science, Ursinus College, Collegeville, PA  19426, {\tt wiswei@ursinus.edu}}}


\begin{abstract}
In this paper, we study the Lusternik--Schnirelmann category of a simplicial map between simplicial complexes, generalizing the simplicial category of a complex to that of a map. Several properties of this new invariant are shown, including its relevance to products and fibrations.   
We relate this category of a map to the classical Lusternik--Schnirelmann category of a map between finite topological spaces.  Finally, we show how the simplicial category of a map may be used to define and study a simplicial version of the essential category weight of J. Strom. 
\end{abstract}


\end{frontmatter}

\section{Introduction}
The Lusternik--Schnirelmann (LS) category of a topological space $X$, denoted $\cat(X)$, was first defined by Lusternik and Schnirelmann in \cite{L-S} for the purpose of studying critical points on manifolds.  Since its humble inception, it has inspired a wide variety of novel techniques in algebraic topology as well as ascertaining many new and interesting numerical invariants of topological spaces.  See, for example, the book \cite{C-L-O-T} devoted to the topic.   

Recently, there has been a growing interest in carrying the LS category into the simplicial or discrete setting \cite{AS,F-M-V,Tanaka16}.  In this paper, we continue the work begun by Fern{\'a}ndez-Ternero et al. in \cite{F-M-V} and \cite{F-M-M-V} by studying the simplicial LS category of a simplicial map.  In addition to continuing the work of these authors, our work extends several results from the classical setting due to Bernstein and Ganea \cite{BG61}, Hardie \cite{H-70}, and Strom. After establishing the necessary background in Section \ref{Background}, we give the definition and basic properties of the simplicial LS category of a simplicial map in Section \ref{Basic properties}. The rest of Section \ref{Simplicial category of a map} is devoted to studying the simplicial category of a product map as well as a simplicial fibration.

Finally, we apply our theory in Section \ref{Applications} by giving a relationship between the simplicial category of a map and the classical category of the corresponding finite space. In addition, we continue to illustrate the phenomena between  the simplicial and classical LS category by defining and studying basic properties of the so-called essential category of a map, following the work of Strom \cite{Strom97, Strom98}.

\section{Background}\label{Background}

In this Section we establish the necessary background that will be used throughout this paper.  All simplicial complexes are assumed to be finite and (edge-path) connected, and all maps between simplicial complexes are assumed to be simplicial maps.  Following the convention in algebraic topology, with domain $K$ and codomain $L$ understood from context, we will use $*$ to denote the simplicial map $*\colon K \to L$ that sends everything to some vertex in $L$. 

We begin by defining the simplicial maps and our notion of equivalence. Our references for the fundamentals of simplicial complexes and maps between them is \cite{S-66} and \cite{BThesis}.

\begin{defn} Let $K$ and $L$ be (abstract) simplicial complexes.  A subset $U\subseteq K$ which is also a simplicial complex is called a \textbf{subcomplex}.
A \textbf{simplicial map} $f \colon K \to L$ is a function from the vertex set of $K$ to the vertex set of $L$ such that if $\sigma$ is a simplex of $K$, then $f(\sigma)$ is a simplex of $L$. Two maps $f,g\colon K\to L$ are said to be \textbf{contiguous}, denoted $f\sim_c g$,  if whenever $\sigma$ is a simplex in $K$, $f(\sigma)\cup g(\sigma)$ is a simplex in $L$. Furthermore, we say that $f$ and $g$  are in the same \textbf{contiguity class}, denoted $f\sim g$ if there exists a sequence of simplicial maps $f_0,f_1,\ldots,f_m\colon  K\to L$ such that $f=f_0\sim_c f_1\sim_c\ldots \sim_c f_{m-1}\sim_c f_m=g$. If $f\sim *$, we say that $f$ is in the \textbf{contiguity class of a vertex.}
\end{defn}

\begin{defn}
Simplicial complexes $K$ and $L$ are said to have the same \textbf{strong homotopy type}, denoted $K\sim L$, if there are simplicial maps $f\colon K\to L$ and  $g\colon L\to K$ such that $g\circ f\sim \id_{K}$ and $f\circ g\sim \id_L$.  In this case, we say that $f$ is a \textbf{strong homotopy equivalence}.
\end{defn}

Strong homotopy was defined and studied by Barmak and Minian in \cite{B-M-12}.  It was then used in \cite{F-M-V} to define the notion of the simplicial Lusternik--Schnirelmann category of a simplicial complex.




\begin{defn}
Let $K$ be a simplicial complex. A subcomplex $U\subseteq K$ is \textbf{categorical} if $i_U\sim *$ where $i_U\colon U\rightarrow K$ is the inclusion and $*\colon U\to K$ is the map sending everything in $U$ to a single vertex. The \textbf{simplicial Lusternik--Schnirelmann category, simplicial LS category, or simplicial category of $K$}, denoted $\scat(K)$, is the least integer $n$ such that $K$ can be covered by categorical subcomplexes $U_0,U_1,\ldots,U_n$. We call $U_0, U_1, \ldots, U_n$ a \textbf{categorical cover of} $K$.
\end{defn}

We also have the notion of a simplicial ``geometric'' category which will play a role in Section \ref{Factorization}.

\begin{defn} The \textbf{simplicial geometric category of $K$}, denoted $\gscat(K)$, is the least integer $m\geq 0$ such that there exists subcomplexes $U_0, U_1, \ldots, U_m$ covering $K$ with $\id_{U_j}\sim *$ for all $0\leq j \leq m$. 
\end{defn}

The following relationship between $\scat$ and $\gscat$ is immediate. 

\begin{prop}\label{scat less gscat}\cite[Proposition 4.2]{F-M-V} For every simplicial complex $K$, $\scat(K)\leq \gscat(K)$.
\end{prop}

\section{Simplicial category of a map}\label{Simplicial category of a map}

This Section is devoted to investigating the simplicial category of a simplicial map, our main object of study.  We note that many of these results are analogues of known results in the classical case (see the book \cite{C-L-O-T}) as well as extensions of results in \cite{F-M-V} and \cite{F-M-M-V} for simplicial complexes.  

\subsection{Basic properties}\label{Basic properties}

We begin with the definition of the simplicial category of a simplicial map.

\begin{defn}
Let $f\colon K\rightarrow L$ be a simplicial map. The \textbf{simplicial Lusternik--Schnirelmann category, simplicial LS category, or simplicial category of $f$}, denoted $\scat(f)$, is the least integer $n$ such that $K$ can be covered by subcomplexes $U_0, U_1,\ldots ,U_n$ with $f|_{U_j}\sim *$ for all $0\leq j\leq n$. We call $U_0, \ldots, U_n$ a \textbf{categorical cover of} $f$.
\end{defn}

Immediately, we see that this generalizes the definition of simplicial category of a simplicial complex. 

\begin{prop}\label{id same} 
Let $K$ be a simplicial complex. Then
$\scat(\id_K)=\scat(K)$. 
\end{prop}

\begin{proof} 
Assume that $\scat(\id_K) =n$. Then $K$  can be covered by subcomplexes $U_0,U_1,\ldots,U_n$ with $\id_K|_{U_j}\sim *$.  Since $\id_K|_{U_j}=i_{U_j}$ and $\id_K|_{U_j}\sim *$, $i_{U_j}\sim *$ so that $U_0,U_1,\ldots, U_n$ forms a categorical cover of $K$. The same reasoning shows that a categorical cover of $K$ yields a categorical cover of $\id_K$. Thus $\scat(\id_K)=\scat(K)$.
\end{proof}

As in the classical case, the simplicial category of a composition bounds below the simplicial category of either map.

\begin{prop}\label{comp prop}
If $f\colon K\to L$ and $g\colon L\to M$, then 

$$\scat(g\circ f)\leq \min\{\scat(g),\scat(f)\}.$$
\end{prop}

\begin{proof}
Assume that $f\colon K\to L$ and  $g\colon L\to M$. The intention is to show that $\scat(g\circ f)\leq \min\{\scat(g),\scat(f)\}$. We first show that $\scat(g\circ f) \leq \scat(f)$. Write $\scat(f)=n$ so that there exists $U_0,U_1,\ldots,U_n\subseteq K$ covering $K$ with $f|_{U_j}\sim *$. We claim that $(g\circ f)|_{U_j}\sim *$. Observe that $(g\circ f)|_{U_j}=g\circ (f|_{U_j})\sim g\circ *\sim*$. Thus, $\scat(g\circ f) \leq \scat(f)=n$.\\

Now write $\scat(g)=m$. Then there exists $V_0,V_1,\ldots,V_m\subseteq L$ such that $g|_{V_j}\sim *$.  Define $U_j:= f^{-1}(V_j)$ for all $0\leq j \leq m$. This clearly covers $K$. Observe that 

$$
\xymatrix{
U_j\ar[rr]^{(g\circ f)|_{U_j}}\ar[d]_{f|_{U_j}}& & M\\
V_j\ar[urr]_{g|_{V_j}} & &
}
$$
commutes up to contiguity class. Since $g|_{V_j}\sim *$, we have $(g\circ f)|_{U_j}\sim *$.
\end{proof}
 
It then easily follows that the simplicial category of a map bounds below the simplicial category of either simplicial complex. 

\begin{prop}\label{map min}
Let $f\colon K\to L$. Then $\scat(f)\leq \min\{\scat(K), \scat(L)\}$. 
\end{prop}

\begin{proof}
Assume that $f\colon K\to L$. The intention is to show that $\scat(f)\leq \min \{\scat(K),\scat(L)\}$. Observe that 
$$
\scat(f)=\scat(f\circ \id_K)\leq \scat(\id_K)=\scat(K)
$$

\noindent by Propositions \ref{comp prop} and \ref{id same}. The same argument shows that $\scat(f)\leq \scat(L)$. Thus $\scat(f)\leq \min \{\scat(K), \scat(L)\}.$
\end{proof}

\begin{rem} It is easy to construct examples where the above inequalities are strict by letting $f=*$ and $K$ have arbitrarily large category.  
\end{rem}

We now show that simplicial category is well-defined up to contiguity class.  It is a matter of showing it is well-defined up to contiguity followed by formal induction.   

\begin{prop}\label{well cont}
Let $f,g\colon K\to L$. If $f\sim g$, then $\scat(f)=\scat(g)$. 
\end{prop}

\begin{proof}
Assume without loss of generality that $f\sim_c g$ and let $U_0, U_1, \ldots, U_n$ be a categorical covering of $g$ so that $g|_{U_j}\sim *$. By definition of contiguity, $f(\sigma)\cup g(\sigma)$ is a simplex in $L$, so $f|_{U_j}(\sigma)\cup g|_{U_j}(\sigma)$ is also a simplex in $L$. Thus, $f|_{U_j}\sim_c g|_{U_j}\sim *$ and $\scat(f)\leq \scat(g)$.  The other direction is identical. 
\end{proof}

The following Corollary gives an alternative proof to and generalizes Theorem 4.3 of \cite{F-M-V}.

\begin{cor}
If $f\colon K\to L$ is a strong homotopy equivalence, then $\scat(f)=\scat(K)=\scat(L).$
\end{cor}

\begin{proof}
Since $f$ is a strong homotopy equivalence, there exists $g\colon L\to K$ such that $g\circ f\sim \id_k$ and $f\circ g \sim \id_L$. We have
\begin{alignat*}{2}
\scat(K)&=\scat(\id_K) && \qquad \text{by Proposition \ref{id same}}\\
& =\scat(g\circ f) && \qquad \text{by Proposition \ref{well cont}}\\
& \leq \scat(f) && \qquad \text{by Proposition \ref{comp prop}}\\
&\leq \scat(K) && \qquad \text{by Proposition \ref{map min}.}\\
\end{alignat*}

It follows that $\scat(K)=\scat(f)$.  The exact same argument shows that $\scat(L)=\scat(f)$.
\end{proof}

We may also relate the simplicial category of a map to the subspace simplicial category.  First, a definition.

\begin{defn}
Let $A\subseteq K$ be a subcomplex. The \textbf{subspace simplicial category of $A$ in $K$}, denoted $\scat_K(A)$, is the least integer $n$ such that there exists $U_0, U_1, \ldots, U_n\subseteq K$ that cover $A$ and each $U_i$ is categorical in $K$. 
\end{defn}

\begin{prop}
Let $i\colon A\to K$ be the inclusion. Then $\scat(i)=\scat_K(A)$.
\end{prop}

\begin{proof}
Assume that $i\colon A\to K$ is the inclusion, and write $\scat_K(A)=n$.  Then there exists $U_0, \ldots, U_n \subseteq K$ covering $A$ with $i_{U_j}\sim *$ where $i_{U_j}\colon U_j\to K$ is the inclusion.  Then $U_0\cap A, U_1\cap A,\ldots, U_n\cap A$ is a categorical cover of $i$ since $i|_{U_j\cap A}=i_{U_j}\circ a \sim *\circ a\sim *$ where $a\colon U_j\cap A\to U_j$ is the inclusion.  Hence $\scat(i)\leq \scat_K(A)$.  To see the other direction, observe that if $V_j$ is a member of a categorical cover of $i$, then $*\sim i|_{V_j}=i_{V_j}$.  
\end{proof}

\subsection{Factorization}\label{Factorization}

In this Section, we show that $\scat(f)\leq n$ may be characterized by factorizing through a simplicial complex $K'$ with $\gscat(K')\leq n$.  This is an  analogue of the result of Berstein and Ganea \cite[Proposition 1.7]{BG61} as well as a natural generalization of the fact that a continuous map is nullhomotopic if and only if it factors through the cone. First, an easy characterization of a map in the contiguity class of a vertex.

\begin{lem}\label{scat 0}
Let $f\colon K\to L$. Then $\scat(f)=0 \Leftrightarrow f \sim *$.
\end{lem}

\begin{proof}
Assume that $\scat(f)=0$. Then there exists only $U_0$ that covers $K$ i.e.  $U_0=K$. Hence $f=f|_K\sim *$. The backwards direction is Proposition \ref{well cont}.
\end{proof}

Recall that if $K$ is a simplicial complex, we define the \textbf{cone on $K$} by $CK:=\{\sigma, \sigma \cup \{v\} : \sigma \in K\}$ where $v$ is some vertex not in $K$. It is then easy to see that $CK$ has the strong homotopy type of a vertex, whence $\scat(CK)=0$.



\begin{lem}\label{Cone map}
Let $f\colon K\to L$. Then $f\sim * \Leftrightarrow f$ factors through $CK$.
\end{lem}

\begin{proof}
If $f\sim *$, then $f$ clearly factors through $CK$. Suppose that $f$ factors through $CK$ so that
$$
\xymatrix{
 & CK\ar[d]^{h}\\
K\ar[ur]^{i}\ar[r]_{f}& L\\
}
$$
commutes up to contiguity class. Observe that $h\circ i\sim f$ and according to Proposition \ref{comp prop}, $\scat(f)\leq \scat(i)$ while  $\scat(i)\leq \scat(CK)$ by Proposition \ref{map min}. But $\scat(CK)=0$ as noted above. Thus, $\scat(f)=0$ and by Lemma \ref{scat 0}, $f\sim *$.
\end{proof}

The following two lemmas are stated without proof.

\begin{lem}\label{Simplicial Pasting} (Simplicial Pasting Lemma)
Let $U,V\subseteq K$ with $f\colon U\to L$ and $g\colon V\to L$ be simplicial maps. Suppose that $f(v)=g(v)$ whenever $v\in U\cap V$. Then $f\cup g\colon U \cup V \to L$ defined by 
$$
f\cup g :=
\begin{cases}
f(u) & \text{if } u\in U \\
g(u) & \text{if } u\in V
\end{cases}
$$
is a well-defined simplicial map.
\end{lem}

\begin{lem}\label{cont union}
Let $f_1,f_2\colon U\to L$ and $g_1,g_2\colon V\to L$ satisfy the conditions of the Simplicial Pasting Lemma so that $f_1\cup g_1$ and $f_2\cup g_2$ are well-defined. If $f_1\sim f_2$ and $g_1\sim g_2$, then $f_1\cup g_1\sim f_2\cup g_2$.
\end{lem}

We now come to one of the main results of this Section.  It will be utilized several times in Section \ref{Essential simplicial category}.

\begin{thm}\label{scat lift}
Let $f\colon K\to L$. Then $\scat(f)\leq n \Leftrightarrow$ there exists a simplicial complex $K'$ with $\gscat(K')\leq n$ such that 
$$
\xymatrix{
 & K'\ar[d]\\
K\ar[ur]\ar[r]_{f}& L\\
}
$$
commutes up to contiguity class.
\end{thm}

\begin{proof}
First suppose that $\scat(f)\leq n$. Then there exists $U_0, U_1,\ldots, U_n\subseteq K$ such that $f|_{U_j}\sim *$. Define $K':= CU_0 \cup CU_1 \cup \ldots \cup CU_n$, where $CU_j$ is the cone on $U_j$. Clearly $\gscat(K')\leq n$. By Lemma \ref{Cone map}, there are simplicial maps $\ell_j\colon U_j\to CU_j$ and $g_j\colon CU_j\to L$ such that 

$$
\xymatrix{
&CU_j\ar[d]^{g_j}\\
U_j\ar[ur]^{\ell_j}\ar[r]_{f|_{U_j}}& L
}
$$

\noindent commutes up to contiguity class i.e. $f|_{U_j}\sim g_j\circ \ell_j$ for all $0\leq j\leq n.$ Combining all these maps together via Lemma \ref{cont union}, we see that $\left(\bigcup g_j\right)\circ \left(\bigcup \ell_j\right) = \bigcup\left(g_j\circ \ell_j\right) \sim \bigcup f|_{U_j}=f$ so that $f$ factors through $K'$.\\

Now suppose that there is a $K'$ with $\gscat(K') \leq n$ and that
$$
\xymatrix{
&K'\ar[d]^{g}\\
K\ar[ur]^{\ell}\ar[r]_{f}& L
}
$$
commutes up to contiguity class. The intention is to show that $\scat(f)\leq n$. Observe that 
\begin{alignat*}{2}
\scat(f) &=\scat(g\circ \ell) && \qquad \text{by Proposition \ref{well cont}}\\
& \leq \scat(g) && \qquad \text{by Proposition \ref{comp prop}}\\
& \leq \scat(K') && \qquad \text{by Proposition \ref{map min}}\\
& \leq\gscat(K') && \qquad \text{by Proposition \ref{scat less gscat}}\\
& \leq n && \qquad \text{by hypothesis.}
\end{alignat*}
Thus $\scat(f)\leq n.$
\end{proof}

\subsection{Products of maps}\label{Products of maps}

Let $f\colon K\to M$ and $g \colon L \to N$ be simplicial maps.  In this Section, we study the simplicial category of the map $f\times g \colon K\times L \to M \times N$ on the categorical product satisfying the universal property of products (\cite[Definition 4.25]{K-08}).  The simplices of $K\times L$ are defined formally as $(\sigma, \tau)$ for each $\sigma\in K$ and $\tau \in L$ with the relation $(\sigma_1, \tau_1)\subseteq (\sigma_2, \tau_2)$ if and only if $\sigma_1 \subseteq \sigma_2$ and $\tau_1 \subseteq \tau_2$ for $\sigma_1, \sigma_2\in K$ and $\tau_1, \tau_2\in L$.  Then $(f\times g)(\sigma,\tau ):=(f(\sigma), g(\tau))$. \\ 


We first compute the simplicial category of some basic maps involving products. 

\begin{defn}
The \textbf{(simplicial) diagonal map on $K$}, denoted $\Delta\colon K\rightarrow K\times K$, is defined by $\Delta(\sigma):= (\sigma,\sigma)$ for all $\sigma \in K$. The \textbf{(simplicial) projection onto the first coordinate} $p_1\colon K\times L \rightarrow K$ is defined by $p_1(\sigma, \tau):= \sigma$ for all $(\sigma, \tau)\in K\times L$.  The simplicial map $p_2\colon K\times L \to L$ is defined similarly. 
\end{defn}

\begin{prop}\label{diag equal}
If $\Delta \colon K\rightarrow K\times K$ is the diagonal map, then $\scat(\Delta)=\scat(K)$.
\end{prop}

\begin{proof}
According to Proposition \ref{map min}, $\scat(\Delta)\leq \min\{\scat(K),\scat(K\times K)\}$ and therefore $\scat(\Delta)\leq\scat(K)$. On the other hand, observe that
\begin{alignat*}{2}
\scat(K)&=\scat(\id_k) && \qquad \text{by Proposition \ref{id same}}\\
& =\scat(p_1\circ \Delta) && \qquad \text{}\\
& \leq \scat(\Delta) && \qquad \text{by Proposition \ref{comp prop}}\\
\end{alignat*} 
and hence the result.
\end{proof}

\begin{prop}
If $p_1\colon K\times L \rightarrow K$, then $\scat(K)=\scat(p_1)$ (similarly for $p_2$).
\end{prop}

\begin{proof}
Use the same estimates as Proposition \ref{diag equal} along with the strictly commutative diagram
$$
\xymatrix{
K\ar@{^{(}->}[r]\ar[d]_{\id_K} & K\times L\ar[dl]^{p_1}\\
K.
}
$$

\end{proof}

The main result of this Section requires the following Theorem. 

\begin{thm}\cite[Theorem 5.14]{F-M-M-V}\label{complex product}
For simplicial complexes $K$ and $L$, we have $\scat(K\times L)\leq (\scat(K)+1)(\scat(L)+1)-1.$
\end{thm}

\begin{thm}
Let $f\colon K\to M$ and  $g\colon L\to N$ be simplicial maps. Then 
$$\scat(f\times g)\leq(\scat(f)+1)(\scat(g)+1)-1.$$
\end{thm}

\begin{proof}
Write $\scat(f)=n$ and $\scat(g)=m$. By Theorem \ref{scat lift}, there exists $K', L'$ with $\gscat(K')\leq n$ and $\gscat(L')\leq m$ along with lifts $\phi_1, \phi_2$ making the following diagrams commute up to contiguity class:
$$
\xymatrix{
&K'\ar[d]^{\psi_1}&&&L'\ar[d]^{\psi_2}\\
K\ar[ur]^{\phi_1}\ar[r]_{f}& M&&L\ar[ur]^{\phi_2}\ar[r]_{g}& M.
}
$$
Since $f\sim \psi_1\circ \phi_1$ and $g\sim \psi_2\circ \phi_2$,  we have $f\times g\sim (\psi_1 \times \psi_2)\circ (\phi_1\times \phi_2)$.

$$
\xymatrix{
K\times L\ar[d]_{\phi_1\times \phi_2}\ar[drr]^{f\times g}\\
K'\times L\ar[rr]_{\psi_1\times\psi_2}&& M\times N
}
$$

\noindent Hence 

\begin{alignat*}{2}
\scat(f\times g) &=\scat((\psi_1\times\psi_2)\circ(\phi_1\times\phi_2)) && \qquad \text{by Proposition \ref{well cont}}\\
& \leq \scat(\phi_1\times \phi_2) && \qquad \text{by Proposition \ref{comp prop}}\\
& \leq \scat(K'\times L') && \qquad \text{by Proposition \ref{map min}}\\
& \leq (\scat(K')+1)(\scat(L')+1)-1 && \qquad \text{by Theorem \ref{complex product}}\\
& \leq (\gscat(K')+1)(\gscat(L')+1)-1 && \qquad \text{by Proposition \ref{scat less gscat}}\\
& \leq (n+1)(m+1)-1 && \qquad \text{}\\
& = (\scat(f)+1)(\scat(g)+1)-1. && \qquad \text{}\\
\end{alignat*}
\end{proof}

\subsection{Fibrations}\label{Fibrations}

Fibrations for simplicial complexes were recently defined by Fern{\'a}ndez-Ternero et al. \cite{F-M-M-V} and used to prove a result relating the simplicial categories of the complexes in a fibration (see Corollary \ref{original fibration}).  In this Section, we generalize this result by showing it can easily be deduced from a fact about the simplicial category of maps.  We first state the definition of a fibration in our setting.   

\begin{defn}
A simplicial map $E\xrightarrow{p} B$ between path connected simplicial complexes is a \textbf{Hurewicz simplicial fibration}, or simply a fibration, if for any $f,g\colon K\to B$ with $f\sim g$ and for any $\widehat{f}\colon K\to E$ with $p\circ \widehat{f}=f$, there exists a lift $\widehat{g}\colon K\to E$ such that $\widehat{f}\sim \widehat{g}$ and $p\circ \widehat{g}=g$.
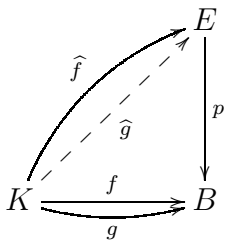
$$
\xymatrix{
&& E\ar[dd]^{p}\\
&\\
K\ar@/^1pc/[uurr]^{\widehat{f}}\ar@{-->}[uurr]_{\widehat{g}}\ar[rr]^{f}\ar@/_/[rr]_{g}&& B
}
$$

\noindent For any $b\in B$, the simplicial complex $F:=p^{-1}(\{b\})$ is the \textbf{fiber} of $p$.
\end{defn}

It has also been shown by the same authors mentioned above that the fiber is well-defined up to strong homotopy type \cite[Theorem 5.2.6]{F-M-M-V}.  We now prove a result relating the simplicial categories of maps in a fibration to the simplicial category of the base space.  The analogous result in the classical case is originally due to Hardie \cite{H-70}.

\begin{thm}\label{map fibration}
Let $F\xrightarrow{i}E\xrightarrow{p}B$ be a fibration. Then $$\scat(E)\leq(\scat(i)+1)(\scat(p)+1)-1.$$
\end{thm}

\begin{proof}
Let $\scat(i)=n$ with cover $V_0,V_1,\ldots,V_n\subseteq F$ and $i|_{V_j}\sim *$, and $\scat(p)=m$ with cover $U_0,U_1,\ldots,U_m\subseteq E$ and $p|_{U_j}\sim *$.  Denoting by $\widehat{p}_j\colon U_j \to E$ the inclusion, we have the following diagram:

$$
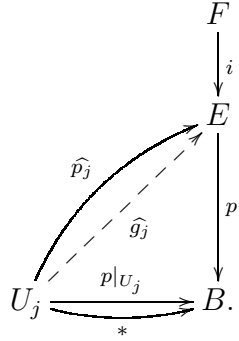
\xymatrix{
&& F \ar[d]^i \\
&& E\ar[dd]^{p}\\
&\\
U_j\ar@/^1pc/[uurr]^{\widehat{p_j}}\ar@{-->}[uurr]_{\widehat{g_j}}\ar[rr]^{p|_{U_j}}\ar@/_/[rr]_{*}&& B.
}
$$

Since $p|_{U_j}\sim *$ and  $p\circ \widehat{p}_j=p|_{U_j}$, the fact that $p$ is a fibration implies that there exists a lift $\widehat{g}_j\colon U_j \to E$ such that $\widehat{g}_j\sim \widehat{p}_j$ and $p\circ \widehat{g}_j=*$.  This latter equation implies that $p(\widehat{g}_j(U_j))=*$ so $\widehat{g}_j(U_j)\subseteq F$.  Define $W_{ij}:=\widehat{g}^{-1}(V_j)\subseteq U_j$.  Then $\bigcup W_{ij}=E$ since $\bigcup \widehat{g}^{-1}(V_i)$ covers $U_j$ for fixed $j$ and all the $U_j$ cover $E$.  It remains to show that $r_{W_{ij}}\sim *$ where $r_{W_{ij}}\colon W_{ij}\to E$ is the inclusion. Let $\widehat{g}_{ij}$ denote the inclusion $\widehat{g}_j$ restricted to $W_{ij}$ into $V_j$. Then the diagram
$$
\xymatrix{
W_{ij}\ar[d]_{\widehat{g}_{ij}} \ar[rr]^{r_{W_{ij}}}&& E\\
V_j \ar[rr]^{i|_{V_j}} && F\ar[u]^i 
}
$$
\noindent commutes up to contiguity class since $\widehat{g}_j\sim \widehat{p}_j$ with $\widehat{p}_j$ the inclusion. But $i|_{V_j}\sim *$ by hypothesis, whence each $W_{ij}$ is categorical in $E$, and the result follows.  
\end{proof}

As a Corollary, we immediately obtain the result of Fern{\'a}ndez-Ternero et al. relating the $\scat$ of the fiber, base space, and total space.

\begin{cor}\label{original fibration}\cite[Theorem 5.2.7]{F-M-M-V} Let $F\xrightarrow{i}E\xrightarrow{p}B$ be a fibration. Then $\scat(E)\leq(\scat(F)+1)(\scat(B)+1)-1$.
\end{cor}

\begin{proof} By Proposition \ref{comp prop}, $\scat(i)\leq \scat(F)$ and $\scat(p)\leq \scat(B)$.  Combining this with Theorem \ref{map fibration}, we see that $\scat(E)\leq (\scat(i)+1)(\scat(p)+1)-1\leq(\scat(F)+1)(\scat(B)+1)-1$. 
\end{proof}






\section{Applications}\label{Applications}

\subsection{Relationship with finite topological spaces}\label{Relationship with finite topological spaces}

We briefly look at the relationship between the simplicial category of a map and the classical category of the corresponding continuous map between finite $T_0$-spaces and vice versa. Beginning with the work of Stong \cite{Stong-66}, finite spaces have recently enjoyed a growing interest. Finite $T_0$-spaces are in bijective correspondence with posets by declaring that $x\leq y$ if and only if $U_x \subseteq U_y$ where $U_x$ is the intersection of all open sets containing $x$.  Conversely, given a poset $(P,\leq)$, we define a basis for a topology on $P$ by $U_x:=\{y\in P: y\leq x\}$.  Furthermore, there is a functor $\chi$ from simplicial complexes to posets (finite $T_0$-spaces) as well as a functor $\mathcal{K}$ from posets to simplicial complexes. 

Let $f\colon K\to L$ be a simplicial map.  The corresponding poset  $\chi(K)$ has vertex set consisting of all the simplices of $K$ while two vertices $\sigma, \tau$ satisfy  $\sigma \leq \tau$ if and only if $\sigma \subseteq \tau$.  In addition, we obtain a continuous function $\chi(f)\colon \chi(K)\to \chi(L)$ defined by $\chi(f)(\sigma):=f(\sigma)$.  Furthermore, $\chi$ takes simplicial maps in the same contiguity class to homotopic maps.

\begin{prop}\label{homotopic maps}\cite[Proposition 2.1.3]{BThesis} Let $f,g\colon K\to L$ be simplicial maps with $f\sim g$.  Then $\chi(f)\simeq \chi(g)$.
\end{prop}

The following Proposition is the analogue for maps found in \cite[Proposition 6.4]{F-M-V}.

\begin{prop}\label{category relations} 
Let $f\colon K\to L$ be a simplicial map. Then $\cat((\chi(f))\leq \scat(f)$.
\end{prop}

\begin{proof} 
Let $\scat(f)=n$. Then there exists $U_0,U_1,\ldots, U_n\subseteq K$ covering $K$ such that $f|_{U_j}\sim *$. By Proposition \ref{homotopic maps}, $\chi(f|_{U_j})\sim \chi(*)=*$. Observe that $\chi (f|_{U_j})(\sigma)=(f|_{U_j})(\sigma)=\chi(f)|_{\chi(U_j)}(\sigma)$ whence $\chi (U_0),\chi (U_1),\ldots,\chi(U_n)$ provides a categorical cover of $\chi(f)$ so that $\cat(\chi(f))\leq\scat(f)$. 
\end{proof}



We now investigate the other direction.  Recall that if $X$ is a finite $T_0$-space, we may construct a corresponding simplicial complex, the \textbf{order poset} denoted $\mathcal{K}(X)$, by defining the vertex set of $\mathcal{K}(X)$ to be the points of $X$ while simplices are given by the non-empty chains in the order of $X$ or, equivalently, the open sets of $X$. Furthermore, if $f\colon X \to Y$ is continuous, the induced map on the order complexes $\mathcal{K}(f)\colon \mathcal{K}(X)\to \mathcal{K}(Y)$ is given by $\mathcal{K}(f)(x):=f(x)$. It can also be shown that if $f,g\colon X \to Y$ are continuous maps between finite spaces with $f\simeq g$, then $\mathcal{K}(f)\sim \mathcal{K}(g)$ \cite[Proposition 2.12]{BThesis}.
%

\begin{prop} Let $f\colon X\to Y$ be a continuous map between finite topological spaces.  Then $\scat(\mathcal{K}(f))\leq \cat(f)$.
\end{prop}

\begin{proof} Let $\cat(f)=n$. Then there are open sets $U_0, U_1,\ldots, U_n \subseteq X$ that cover $X$ such that $f|_{U_j}\simeq *$. Then, as mentioned above, $\mathcal{K}(f|_{U_j})\sim \mathcal{K}(*)$. Since $\mathcal{K}(f|_{U_j})(x)=(f|_{U_j})(x)=\mathcal{K}(f)|_{\mathcal{K}(U_j)}(x)$ and clearly $\mathcal{K}(*)=*$, $\mathcal{K}(U_j)$ forms a categorical cover of $\mathcal{K}(f)$, and the result follows. 
\end{proof}

\subsection{Essential simplicial category}\label{Essential simplicial category}

The essential category weight of a continuous map was defined by Strom \cite{Strom97, Strom98} as a means to study the classical category of a space.  In this Section, we carry over some of the basic results into the simplicial setting.  

\begin{defn}
A simplicial map $f\colon K\to L$ has \textbf{essential simplicial category at least $n$}, denoted $\Es(f)\geq n$, if for every $M$ with $\scat(M)\leq n$ and simplicial map $g\colon M\to K$, we have $f\circ g\sim *$.
\end{defn}

The essential simplicial category is an invariant of the contiguity class of map chosen and furthermore, bounds strictly from below the simplicial category.

\begin{prop}
If $f_0\sim f_1$, then $\Es(f_0)=\Es(f_1).$ Furthermore, if $f\not \sim *$, then $\Es(f)<\scat(K)$.
\end{prop}

\begin{proof}
Write $\Es(f_1)=n$. To show that $\Es(f_0)\geq n$, let $M$ by any simplicial complex with $\scat(M)\leq n$ and $g\colon M\to K$ be any simplicial map.
Since $f_0\sim f_1$, $f_0\circ g\sim f_1\circ g$ and because $\Es(f_1)=n$, $f_1\circ g\sim *$. Hence $\Es(f_0)\geq n$. The other direction is identical.

\bigskip
Now suppose that $f\not \sim *$, and suppose by contradiction that $\Es(f)\geq \scat(K)=n$. Choosing $M=K$ and $g=\id_K$ as above, we see that $f\sim *$ by definition, a contradiction.
\end{proof}

Using our characterization of the simplicial category of a map above (Theorem \ref{scat lift}), we may in turn give alternative characterizations of the essential simplicial category. 

\begin{prop}
Let $f\colon K\to L$. The following are equivalent:
\begin{enumerate}
\item[$1)$] $\Es(f)\geq n$.\\
\item[$2)$] $f\circ g\sim *$ for all  $g\colon M\to K$ with $\scat(g)\leq n$.\\
\item[$3)$] $f\circ g\sim *$ for all  $g\colon M\to K$ with $\gscat(M)\leq n$.
\end{enumerate}
\end{prop}

\begin{proof}
Suppose that $\Es(f)\geq n$ and let $g\colon M\to K$ with $\scat(g)\leq n$. Then by Theorem \ref{scat lift}, there is an $M'$ with $\scat(M')\leq \gscat(M')\leq n$ and lift making the following diagram commute up to contiguity class:
$$
\xymatrix{
& M'\ar[d]^{h}\\
M\ar[r]_{g} \ar[ur]^{g'}& K\ar[r]_{f}& L.
}
$$
Since $\Es(f)\geq n$, $(h\circ g')\circ f\sim *$ so $f\circ g\sim (h\circ g')\circ f\sim *$.\\

Now assume 2) and let $g\colon M\to K$ with $\gscat(M)\leq n$. Applying Theorem \ref{scat lift} to the diagram 
$$
\xymatrix{
& M\ar[d]^{g}\\
M\ar[r]_{g} \ar[ur]^{\id_M}& K\ar[r]_{f}& L
}
$$
we see that $\scat(g)\leq n$. By 2), $f\circ g\sim *$.\\

Finally, suppose that 3) holds, and let $g\colon M\to K$ with $\scat(M)\leq n$. By Proposition \ref{map min}, $\scat(g)\leq \scat(M) \leq n$ so by Theorem \ref{scat lift}, $g$ factors through a simplicial complex $M'$ with $\gscat(M')\leq n$. By 3), $f\circ g \sim *$ and $\Es(f)\geq n$.
\end{proof}

Finally, the essential simplicial category of a composition bounds above the product of the essential simplicial category of each of the maps.

\begin{prop}\label{essential product}\cite[Cf. Theorem 9]{Strom98}
Let $f\colon K\to L$ and  $g\colon L \to M$. Then $\Es(g\circ f)\geq (\Es(g)+1)(\Es(f)+1)-1$.
\end{prop}

\begin{proof}
Write $\Es(f)=p$ and $\Es(g)=q$ and let $h\colon Z\to M$ be any map with $\scat(Z)\leq (p+1)\cdot (q+1)-1$. Write $Z=U_0 \cup U_1\cup \ldots \cup U_q$ with $\scat(U_i)\leq p$. Since $\Es(f)=p$, we have $(f\circ h)|_{U_j}\sim *$ and  

$$
\xymatrix{
& & Z'\ar[d]^{i}\\
Z\ar[r]_{h} \ar[urr]^{}& K\ar[r]_{f}& L\ar[r]_{g}& M
}
$$

\noindent where $Z'=Z\cup (\amalg C U_i)$. Clearly $\scat(Z')\leq q$ and since $\Es(g)=q, g\circ i \sim *$, whence $(g\circ f)\circ h \sim *$. We conclude that $\Es(g\circ f)\geq (\Es(g)+1)(\Es(f)+1)-1$.
\end{proof}

The following is then immediate.

\begin{cor} $\Es(g \circ f)\geq \Es(f)$ and $\Es(g \circ f)\geq \Es(g)$.
\end{cor}

\section{Future directions}\label{Future directions}

The simplicial category of a simplicial map is a natural generalization of the recently defined simplicial category of a simplicial complex. There are many directions that the work in this paper, and simplicial LS category in general, may go.  

One important use of the category of a map was in studying special cases and  examples which pertained to the Ganea conjecture. Could the simplicial category of a map be used to study a simplicial version of the Ganea conjecture and related questions?  Currently, because different triangulations of spheres may have different strong homotopy type, we do not even know the simplicial category of all simplicial complexes homeomorphic to a sphere.  

In \cite{F-M-M-V}, the authors give a Whitehead definition of simplicial category in terms of factoring through a simplicial fat wedge. While it was shown that this bounds above the simplicial category, the two strong homotopy invariants, unlike the classical case, are not in general equal. One could give a Whitehead definition of the simplicial category of a map and study it as well as its relationship to the simplicial category of a map defined in this paper.   

We have only scratched the surface of the use of the essential simplicial category.  For example, given a complex $K$, can one construct maps $f_K$ such that $\scat(K=\Es(f_K)+1$? Strom has used cohomology to produce relations with essential category in the smooth case. How can cohomology provide information to estimate and compute simplicial category?

\bibliographystyle{amsplain}
\bibliography{EssentialCat}
\end{document}